\pdfoutput=1
\documentclass[12pt]{article}
\usepackage{amsmath,amsthm,amssymb,color,enumerate}
\usepackage{graphicx}
\usepackage[UKenglish]{babel}
\usepackage{lmodern}
\usepackage{microtype}
\usepackage{mathtools}
\usepackage{esint}
\usepackage{comment}
\usepackage{hyperref}
\usepackage{authblk}
\usepackage[utf8]{inputenc}
\usepackage{enumitem}

\mathtoolsset{showonlyrefs}

\theoremstyle{plain}
\newtheorem{theorem}{Theorem}

\newtheorem{lemma}[theorem]{Lemma}
\newtheorem{corollary}[theorem]{Corollary}

\newtheorem{proposition}[theorem]{Proposition}

\theoremstyle{definition}

\newtheorem{example}[theorem]{Example}

\theoremstyle{remark}
\newtheorem{remark}[theorem]{Remark}

\newcommand*{\R}{\mathbb{R}}

\newcommand*{\N}{\mathbb{N}}
\newcommand*{\Z}{\mathbb{Z}}

\newcommand*{\eps}{\varepsilon}
\newcommand*{\doo}{\partial}
\newcommand*{\ol}[1]{\overline{#1}}
\newcommand{\sulut}[1]{\left( #1 \right)}
\newcommand{\joukko}[1]{\left\{ #1 \right\}}
\newcommand{\abs}[1]{\left\lvert #1 \right\rvert}

\newcommand{\norm}[1]{\left\| #1 \right\|}
\newcommand{\der}{\mathrm{d}}

\newcommand{\ip}[2]{\left\langle#1,#2\right\rangle}

\newcommand*{\E}{\mathbb{E}}

\DeclareMathOperator{\dive}{div}

\DeclareMathOperator*{\esssup}{ess \, sup}
\DeclareMathOperator*{\essinf}{ess \, inf}

\DeclareMathOperator{\Span}{span}

\makeatletter
\newcommand*{\mint}[1]{%
  \mint@l{#1}{}%
}
\newcommand*{\mint@l}[2]{%
  \@ifnextchar\limits{%
    \mint@l{#1}%
  }{%
    \@ifnextchar\nolimits{%
      \mint@l{#1}%
    }{%
      \@ifnextchar\displaylimits{%
        \mint@l{#1}%
      }{%
        \mint@s{#2}{#1}%
      }%
    }%
  }%
}
\newcommand*{\mint@s}[2]{%
  \@ifnextchar_{%
    \mint@sub{#1}{#2}%
  }{%
    \@ifnextchar^{%
      \mint@sup{#1}{#2}%
    }{%
      \mint@{#1}{#2}{}{}%
    }%
  }%
}
\def\mint@sub#1#2_#3{%
  \@ifnextchar^{%
    \mint@sub@sup{#1}{#2}{#3}%
  }{%
    \mint@{#1}{#2}{#3}{}%
  }%
}
\def\mint@sup#1#2^#3{%
  \@ifnextchar_{%
    \mint@sup@sub{#1}{#2}{#3}%
  }{%
    \mint@{#1}{#2}{}{#3}%
  }%
}
\def\mint@sub@sup#1#2#3^#4{%
  \mint@{#1}{#2}{#3}{#4}%
}
\def\mint@sup@sub#1#2#3_#4{%
  \mint@{#1}{#2}{#4}{#3}%
}
\newcommand*{\mint@}[4]{%
  \mathop{}%
  \mkern-\thinmuskip
  \mathchoice{%
    \mint@@{#1}{#2}{#3}{#4}%
        \displaystyle\textstyle\scriptstyle
  }{%
    \mint@@{#1}{#2}{#3}{#4}%
        \textstyle\scriptstyle\scriptstyle
  }{%
    \mint@@{#1}{#2}{#3}{#4}%
        \scriptstyle\scriptscriptstyle\scriptscriptstyle
  }{%
    \mint@@{#1}{#2}{#3}{#4}%
        \scriptscriptstyle\scriptscriptstyle\scriptscriptstyle
  }%
  \mkern-\thinmuskip
  \int#1%
  \ifx\\#3\\\else_{#3}\fi
  \ifx\\#4\\\else^{#4}\fi  
}
\newcommand*{\mint@@}[7]{%
  \begingroup
    \sbox0{$#5\int\m@th$}%
    \sbox2{$#5\int_{}\m@th$}%
    \dimen2=\wd0 %
    \let\mint@limits=#1\relax
    \ifx\mint@limits\relax
      \sbox4{$#5\int_{\kern1sp}^{\kern1sp}\m@th$}%
      \ifdim\wd4>\wd2 %
        \let\mint@limits=\nolimits
      \else
        \let\mint@limits=\limits
      \fi
    \fi
    \ifx\mint@limits\displaylimits
      \ifx#5\displaystyle
        \let\mint@limits=\limits
      \fi
    \fi
    \ifx\mint@limits\limits
      \sbox0{$#7#3\m@th$}%
      \sbox2{$#7#4\m@th$}%
      \ifdim\wd0>\dimen2 %
        \dimen2=\wd0 %
      \fi
      \ifdim\wd2>\dimen2 %
        \dimen2=\wd2 %
      \fi
    \fi
    \rlap{%
      $#5%
        \vcenter{%
          \hbox to\dimen2{%
            \hss
            $#6{#2}\m@th$%
            \hss
          }%
        }%
      $%
    }%
  \endgroup
}

\setlist{nolistsep} 

\title{Variable exponent Calderón's problem in one dimension}

\author{Tommi Brander \\ tommi.brander@ntno.no}
\affil{Technical University of Denmark, Department of Applied Mathematics and Computer Science}

\author{David Winterrose \\ dawin@dtu.dk}
\affil{Technical University of Denmark, Department of Applied Mathematics and Computer Science}

\begin{document}

\maketitle

\begin{abstract}
We consider one-dimensional Calder\'on's problem for the variable exponent~$p\sulut{\cdot}$-Laplace equation and find out that more can be seen than in the constant exponent case.
The problem is to recover an unknown weight (conductivity) in the weighted $p(\cdot)$-Laplace equation from Dirichlet and Neumann data of solutions.
We give a constructive and local uniqueness proof for conductivities in $L^\infty$ restricted to the coarsest sigma-algebra that makes the exponent~$p\sulut{\cdot}$ measurable.
\end{abstract}

\paragraph{Keywords} Calderón's problem, inverse problem, variable exponent, non-standard growth, elliptic equation, quasilinear equation

\paragraph{Mathematics subject classification} 35R30 (primary) 34A55, 35J92, 35J62, 35J70, 46N20, 34B15 (secondary)




\section{Introduction}

Calder\'on's problem~\cite{Calderon:1980} asks if an electric conductivity~$\gamma$ in an object~$\Omega$ can be reconstructed from boundary measurements of current and voltage, which are given by the Dirichlet-to-Neumann map (DN map) $u|_{\doo \Omega} \to \gamma \nabla u \cdot \nu |_{\doo \Omega}$, where $\nu$ is the unit outer normal.
In one dimension, the answer to Calder\'on's problem is negative; only the resistance, that is, the total resistivity $\int_{I} \gamma^{-1} \der x$, can be recovered~\cite[section 1.1]{Feldman:Salo:Uhlmann}, where $I \subset \R$ is an open bounded interval.
A similar result holds for $p$-Calder\'on's problem, where the forward model is given by the weighted $p$-Laplace equation $-\dive \sulut{\gamma \abs{u'}^{p-2}u'} = 0$: it is only possible to recover the value of the integral
\begin{equation}
\int_I \gamma^{1/(1-p)} \der x
\end{equation}
from the DN~map~\cite[theorem 2.2]{Brander:2016:apr}.
This is true for any constant $1 < p < \infty$, and is also a special case of corollary~\ref{lemma:av_sigma} in this paper.
We describe what can be recovered in the case of a variable exponent $p(\cdot)$.
This is the first investigation of an inverse problem related to the variable exponent $p(\cdot)$-Laplacian.
The problem in the constant exponent case was introduced by Salo and Zhong~\cite{Salo:Zhong:2012} in 2012, after which other theoretical results have been published~\cite{Brander:2016:jan,Brander:Ilmavirta:Kar:2018,Brander:Kar:Salo:2015,Brander:Harrach:Kar:Salo:2018,Guo:Kar:Salo:2016,Kar:Wang}, as well as a numerical study~\cite{Hannukainen:Hyvonen:Mustonen:2019}.
The works of Sun and others consider the problem of recovering the dependence of $A$ in
\begin{equation}
\dive \sulut{A(x,u,\nabla u) \nabla u} = 0
\end{equation}
on all of $x,u$, and $\nabla u$, but they either do not admit degenerate equations~\cite{Munoz:Uhlmann:2018,Sun:1996} or assume the equivalent of constant $p$~\cite{Hervas:Sun:2002}.

One can think of the variable exponent conductivity equation
\begin{equation}
    -\dive \sulut{\gamma \abs{\nabla u}^{p(x)-2}\nabla u} = 0
\end{equation}
as arising from a non-linear Ohm's law
\begin{equation}
    -j = \gamma \abs{\nabla u}^{p(x)-2}\nabla u,
\end{equation}
which has a power law dependence between the current $j$ and the gradient of the electric potential~$u$ at every point, but where the exponent in the power law varies from point to point.
We use this terminology and intuition in the article.
An example of a power-law type Ohm's law is certain polycrystalline compounds near the transition to superconductivity~\cite{Bueno:Longo:Varela:2008,Dubson:Herbert:Calabrese:Harris:Patton:Garland:1988}, where the exponent~$p$ is a function of temperature.
However, it is not clear how relevant electrical impedance tomography of such materials would be.

\subsection{Results}

We use two approaches.
The first is to consider the limit of the Dirichlet-to-Neumann map (hereafter DN~map) as the difference of Dirichlet boundary values grows without bound or approaches zero; this gives information about the value of the conductivity at the maximum or minimum of $p$, when the extreme value is reached on a set of positive measure.

The second is to consider the DN~map as a dual pairing in an $L^2$ space of essentially the conductivity and another function, and then determine what one can say about the conductivity based on this information.
This gives uniqueness for the conductivity in $L^\infty$ restricted to the coarsest sigma-algebra that makes $p\sulut{\cdot}$ measurable.
We present two proofs. The simpler one requires knowing the full DN~map and gives a non-constructive proof.
The other proof is constructive and requires knowledge of the DN~map only on an open set, but requires working with fairly explicit formulae for arbitrary order derivatives of composite functions and inverse functions.
This approach is similar to a classical moment problem~\cite{Schmudgen:2017}; whereas the Hausdorff moment problem asks if there exists a measure~$\mu$ such that a given sequence $m_n$ satisfies
\begin{equation}
    m_n = \int_I x^n \der \mu(x),
\end{equation}
we ask what kinds of functions~$f$ satisfy
\begin{equation}
    m_n = \int_I \sulut{g(x)}^n f(x) \der x
\end{equation}
for a specific sequence $m_n$ arising from the forward problem, and for functions $g(x) = 1/(p(x)-1)$ that are derived from the variable exponent~$p$.
The results are very similar to what is achieved for one dimensional inverse source problem under attenuation and using multiple frequencies~\cite{Brander:Ilmavirta:Tyni}.
The numerical methods used there could also be applied to the problem discussed in this paper.

In the case of nonconstant~$p$, even though we only choose our input data from an essentially one-dimensional space (difference of the two Dirichlet data points is a real number), we can recover information on an infinite-dimensional space, if the power $p$ is not piecewise constant.
Here the changing non-linear nature of the forward problem makes the inverse problem easier.
Nonlinearity of the forward problem has also been used for advantage in the study of non-linear hyperbolic equations~\cite{deHoop:Uhlmann:Wang:2018,Kurylev:Lassas:Oksanen:Uhlmann:2018}.

Our main theorems follow.
First, let $I$ be equipped with the Lebesgue measure on the Lebesgue sigma-algebra.
Define $\sigma(p)$ to be the sigma-algebra on~$I$ generated by sets of the form $p^{-1}(A)$ where $A \subseteq \R$ is a Borel set in $\R$.
We consider $p$ to be a function, even though we may write $p \in L^r(I)$.
The spaces $L^r(I,\sigma(p))$ are the Lebesgue spaces of almost everywhere equal $\sigma(p)$-measurable functions of finite $L^r$-norm with respect to the Lebesgue measure restricted to $\sigma(p)$.
It is important to emphasize that $L^r(I,\sigma(p))$ is not a true subspace of $L^r(I)$ because $\sigma(p)$ does not contain all null-sets of the Lebesgue sigma-algebra on $I$. However, the map $L^r(I,\sigma(p))\to L^r(I):[g]_{\sigma(p)} \mapsto [g]$ is a well-defined (possibly not surjective) isometry, so we may regard $L^r(I,\sigma(p))$ as a complete and therefore closed subspace, but emphasize that the equivalence classes are different.
This subspace is characterized by the property that every equivalence class has a $\sigma(p)$-measurable representative.
Finally, we define $P \colon L^2(I) \to L^2(I,\sigma(p))$ as the orthogonal projection (conditional expectation in probabilistic terms) onto the closed subspace identified with $L^2(I,\sigma(p))$.

For more on sigma-algebras generated by functions or sets we refer to the book of Dellacherie and Meyer~\cite[definition~5]{Dellacherie:Meyer:1978}.
\begin{theorem}
\label{thm:characterization}
Consider an open bounded interval~$I$.
Suppose the conductivity~$\gamma \in L^\infty_+(I) \subset L^2(I)$, and suppose there exist constants~$p^\pm$ such that almost everywhere $1 < p^- < p(x) < p^+ < \infty$.
Then the nonlinear projection
\begin{equation}
\widetilde P \colon L^2(I) \cap L^\infty_+(I) \to L^2(I,\sigma(p)) \cap L^\infty_+(I,\sigma(p))
\end{equation}
defined by
\begin{equation}
\widetilde P(\gamma)(x) = \sulut{P\sulut{\gamma^{-1/(p-1)}} (x)}^{-\sulut{p(x)-1}}
\end{equation}
 can be reconstructed from the Dirichlet-to-Neumann map.
Furthermore, there exists an arbitrarily small open set of Dirichlet boundary values which are sufficient for reconstructing the projection.
\end{theorem}

\begin{corollary}
Suppose $\gamma_1,\gamma_2 \in L^\infty_+(I)$.
If $p \colon I \to \R$ is a measurable injection and $\widetilde P (\gamma_1) = \widetilde P (\gamma_2)$, then $\gamma_1 = \gamma_2$.
\end{corollary}
\begin{proof}
In this case the projection $P$ is the identity map and the powers in $\widetilde P$ cancel, so $\widetilde P$ is also the identity map.
\end{proof}

\begin{remark}
\label{remark:P_tilde}
The following properties hold:
\begin{itemize}
\item If $\gamma$ is $\sigma(p)$-measurable, then $\widetilde P(\gamma) = \gamma$.
\item We have $\widetilde P \circ \widetilde P = \widetilde P$.
\item If $p \equiv 2$ and $\abs{I} = 1$, then $\widetilde P$ is the harmonic mean.
\item If $p$ is constant, then $\widetilde P(\gamma)$ is a type of average:
\begin{equation}
    \widetilde P (\gamma) = \sulut{\frac{1}{\abs{I}} \int_{I} \gamma^{-1/(p-1)} \der x}^{-(p-1)}.
\end{equation}
\end{itemize}
\end{remark}
\begin{proof}
If $\gamma$ is $\sigma(p)$-measurable, then so is $\gamma^{-1/(p-1)}$.
Hence, the projection~$P\sulut{\gamma^{-1/(p-1)}} = \gamma^{-1/(p-1)}$ and the exponents inside and outside the projection cancel.
This proves the first point.

To prove the middle point, one observes that in $\widetilde P \circ \widetilde P$ the powers between the two projections~$P$ cancel, after which $P \circ P = P$, since $P$ is a projection.

If $p \equiv 2$, then inside the projection~$P$ there is $\gamma^{-1}$.
The sigma-algebra generated by the constant function~$p \equiv 2$ is the trivial one; that is, $\sigma(p) = \joukko{ \emptyset, I }$.
The only functions measurable with respect to this sigma-algebra are constant functions, whence $P\sulut{\gamma^{-1}}$ must be a constant function; call the constant~$C$.
Since integrals over $\sigma(p)$-measurable sets are conserved by the projection, we have
\begin{equation}
\int_I C \der x = \int_I \gamma^{-1} \der x \iff C = \frac{1}{\abs{I}} \int_I \gamma^{-1} \der x.
\end{equation}
Raising this to the power $-(p-1) \equiv -1$ gives the third claim, and the proof of the fourth claim is similar.
\end{proof}

This theorem is proven in a non-constructive way in section~\ref{sec:characterization} using the multiplicative system theorem (theorem~\ref{thm:multisystem}).
It turns out that the finite linear combinations of functions $K^{p(x)/(p(x)-1)}$ indexed by the real numbers $K \ge 0$ are dense in the space of  functions that are both $\sigma(p)$-measurable and square integrable.
A constructive and local proof is developed in section~\ref{sec:derivatives}.
The constructive proof requires explicit formulae for higher order derivatives of composite functions (Faà di Bruno's formula~\cite{Johnson:2002}) and inverse functions~\cite{Kaneiwa:2016}.
Aside from the difficulty of calculating the derivatives, the proof is quite similar to a related proof of Brander, Ilmavirta and Tyni~\cite{Brander:Ilmavirta:Tyni}.

The next theorem is proven in section~\ref{sec:limits}.
The proof is constructive, easy to implement numerically and reasonably elementary.
\begin{theorem}
\label{thm:pos_meas}
Consider an open bounded interval~$I$.
Suppose the conductivity~$\gamma \in L^\infty(I)$ is almost everywhere bounded from below by a positive constant.
Suppose further that the exponent $p(\cdot)$ reaches its essential minimum~$p^-$ on a set of positive measure~$Q_+$ and essential maximum~$p^+$ on a set of positive measure~$Q_-$.
Then the DN map~$\Lambda_\gamma$, suitably scaled, gives
\begin{align}
\lim_{m \to \infty} \ol K_m^{-p^-/(p^- - 1)} \frac{1}{\abs{Q_+}} \Lambda_\gamma(m) & = \sulut{\mint{-}_{Q_+} \sulut{\gamma(x)}^{-1/\sulut{p(x)-1}} \der x}^{-\sulut{p^- -1}} \text{ and} \\
\lim_{m \to 0} \ol K_m^{-p^+/(p^+ - 1)} \frac{1}{\abs{Q_-}} \Lambda_\gamma(m) & = \sulut{\mint{-}_{Q_-} \sulut{\gamma(x)}^{-1/\sulut{p(x)-1}} \der x}^{-\sulut{p^+ -1}},
\end{align}
where $m$ is the difference between the Dirichlet boundary values and $\ol K_m$ can be computed from knowledge of $m$ and $p(\cdot)$.
\end{theorem}
This theorem allows the recovery of the average of the conductivity to a known power over the set where the exponent is largest or smallest.
The situation is far more delicate if the exponent does not reach its maximum/minimum or does so in a set of zero measure.
The behaviour seems to depend on $\abs{\joukko{x \in I ; p(x) = a}}$ as $a \to p^\pm$.
We do not investigate the matter in more detail in this article.

\subsection*{Acknowledgements}
T.B.\ was funded by grant no.\ 4002-00123 from the Danish Council for Independent Research | Natural Sciences.
We would like to thank Kim Knudsen for several discussions and Nathaniel Eldredge for graciously providing a key result\footnote{\url{https://mathoverflow.net/a/292978/}}.
We would also like to thank the anonymous referees for their valuable input, which improved this article significantly.

\section{Forward problem}
\label{sec:forward}

In this section we discuss the existence and uniqueness for the forward problem in general dimension, and define the voltage-to-current or DN~map, also in general dimension.
We specialize to the one-dimensional inverse problem in the following sections.
Before proceeding, we define the variable exponent Lebesgue space $L^p(\Omega$), with $\Omega \subset \R^d$ a bounded open set and $d \geq 1$.
The variable exponent Sobolev spaces are defined in terms of it in the usual way.
We assume throughout that $p \colon \Omega \to [1,\infty]$ is measurable and bounded away from one and infinity.
Then, following the book of Diening, Harjulehto, Hästö and R\r{u}\v{z}i\v{c}ka~\cite[sections 2 and 3]{Diening:Harjulehto:Hasto:Ruzicka:2011},
\begin{equation}
    L^p(\Omega) = \joukko{f \colon \Omega \to \R \text{ measurable } ; \lim_{\lambda \to 0} \int_\Omega \frac{1}{p(x)} \abs{\lambda f(x)}^{p(x)} \der x = 0 },
\end{equation}
where functions which agree almost everywhere are considered identical, and
\begin{equation}
\norm{f}_{L^p(\Omega)} = \inf \joukko{\lambda > 0 ; \int_\Omega \frac{1}{p(x)} \abs{\frac{f(x)}{\lambda}}^{p(x)} \der x \le 1}.
\end{equation}

Consider a weight function, or conductivity,
\begin{equation}
\gamma \in L_+^\infty (\Omega) = \joukko{f \in L^\infty(\Omega) ; \essinf f > 0}.
\end{equation}
The Dirichlet problem for the weighted variable exponent $p(\cdot)$-Laplacian is
\begin{align}
\dive \sulut{ \gamma \abs{\nabla u}^{p(x)-2} \nabla u } = 0 &\text{ in } \Omega \\
u = f &\text{ on } \doo \Omega.
\end{align}
Uniqueness and existence for the variable exponent function has been investigated in variable exponent Sobolev spaces~\cite{Diening:Harjulehto:Hasto:Ruzicka:2011}, though one often considers the equation
\begin{equation}
\dive \sulut{p(x) \abs{\nabla u}^{p(x)-2} \nabla u } = 0
\end{equation}
as the basic example~\cite[section 2]{Harjulehto:Hasto:Le:Nuortio:2010}.
That equation arises from minimizing the functional
\begin{equation}
\label{eq:abnormal_energy}
u \mapsto \int_\Omega \abs{\nabla u}^{p(x)} \der x,
\end{equation}
while we prefer to work with the energy
\begin{equation}
\label{eq:normal_energy}
u \mapsto \int_\Omega \frac{\gamma}{p(x)} \abs{\nabla u}^{p(x)} \der x.
\end{equation}
Since the function $x \mapsto \frac{1}{p}\abs{x}^p$ is the convex conjugate (by the Legendre-Fenchel transform) of $x \mapsto \frac{1}{q}\abs{x}^{q}$, where $q$ is the H{ö}lder conjugate of $p$, we see the division by $p(x)$ as natural.

Calderón's problem asks one to recover the conductivity~$\gamma$ from the DN~map~$\Lambda_\gamma$, which, in the strong form, is given by the formula
\begin{equation}
\Lambda_\gamma (f) = \gamma \abs{\nabla u}^{p-2}\nabla u \cdot \nu|_{\doo \Omega},
\end{equation}
where $u|_{\doo \Omega}=f$, 
i.e.\ the input is the potential or Dirichlet boundary value and the output is the current flowing out of the domain.
This definition may fail for irregular solutions~$u$ or domains~$\Omega$.
One typically uses the weak DN~map instead, which we derive next.
We remark that regardless of the energy/equation we use, the energy of the equation and its weak DN~map are different.
By formally integrating by parts, starting from the strong definition of DN~map and multiplying by a test function $v$ with $v|_{\doo \Omega} = g$, we get:
\begin{equation}
\begin{split}
&\int_{\doo \Omega} \gamma \abs{\nabla u}^{p(x)-2} \nabla u \cdot \nu g \der S(x) \\
= & \int_{\Omega} \dive\sulut{\gamma \abs{\nabla u}^{p(x)-2} \nabla u} v \der x + \int_{\Omega} \gamma \abs{\nabla u}^{p(x)-2} \nabla u \cdot \nabla v \der x  \\
= & \int_{\Omega} \gamma \abs{\nabla u}^{p(x)-2} \nabla u \cdot \nabla v \der x.
\end{split}
\end{equation}
If we choose $g=f$ and $v = u$, then we have
\begin{equation}
\ip{\Lambda_\gamma \sulut{f}}{f} = \int_{\Omega} \gamma \abs{\nabla u}^{p(x)} \der x.
\end{equation}
We take this ``quadratic'' form as the definition of the DN~map.
It can be defined as a functional on $W^{1,p}(\Omega)/W^{1,p}_0(\Omega)$, where $W^{1,p}_0$ is the closure of the space of $W^{1,p}$ functions with compact support~\cite[section 8.1]{Diening:Harjulehto:Hasto:Ruzicka:2011}.

\begin{lemma}
\label{lemma:unique_existence}
Suppose $1 < p^- \leq p(x) \leq p^+ < \infty$ almost everywhere, and that $\Omega \subset \R^d$, $d \in \Z_+$, is a bounded open set that supports the Poincaré inequality.
Consider boundary values $f \in W^{1,p(\cdot)} \cap L^\infty(\Omega)$.
Then there exist unique minimizers in $W^{1,p(\cdot)} \cap L^\infty(\Omega) + f$ to the energies~\eqref{eq:normal_energy} and \eqref{eq:abnormal_energy}.
\end{lemma}
\begin{proof}
The proof follows the direct method in calculus of variations.
The variable exponent Sobolev space is a reflexive Banach space~\cite[theorem 8.1.6]{Diening:Harjulehto:Hasto:Ruzicka:2011} and the functional is convex, since $t \mapsto c t^p$ is convex for all $p \geq 1$ and $c \geq 0$.
The energies are lower semicontinuous~\cite[theorem 3.2.9 and section 3.2]{Diening:Harjulehto:Hasto:Ruzicka:2011}.
Coercivity of the functional requires the Poincaré inequality with $p\equiv 1$ (since we assume bounded boundary values).
Therefore the functionals have unique minimizers.
\end{proof}

This lemma holds for any bounded, open interval in $\R$.
More generally, 1-Poincaré inequality is satisfied for example in John domains~\cite[section 8.2]{Diening:Harjulehto:Hasto:Ruzicka:2011}, and in particular in Lipschitz domains.

\section{Recovering conductivity}
\label{sec:1d}

We write $\Omega = I = \; ]a,b[ \; \subset \R$.
Assume $1 < p^- \le p(x) \le p^+ < \infty$ almost everywhere on $I$, and that $p$ is a bounded measurable function.
In fact, define
\begin{equation}
p^- = \essinf_{x \in I} p(x) \text{ and } p^+ = \esssup_{x \in I} p(x).
\end{equation}
Later, we will use similar notations for the essential supremum and infimum of another exponent $q(\cdot)$, which is the conjugate Hölder exponent to $p(\cdot)$.

Then, at least formally, the forward problem is
\begin{equation}
\begin{cases}
-\sulut{\gamma \abs{u'}^{p(x)-2}u'}' = 0 \text{ in } I \\
u (a) = A \\
u(b) = B.
\end{cases}
\end{equation}
We suppose $A \leq B$, whence there should exist a constant~$K \ge 0$ such that, for almost every~$x$,
\begin{equation}
\begin{split}
&\gamma \sulut{u'}^{p(x)-1} = K \\
\iff & u' = \sulut{K/\gamma}^{1/(p(x)-1)}
\end{split}
\end{equation}
and hence
\begin{equation}
u(x) = A + \int_{a}^x \sulut{K/\gamma(s)}^{1/\sulut{p(s)-1}} \der s.
\end{equation}
Using $u(b) = B$ we get:
\begin{equation}
\label{eq:K_dirichlet}
\int_a^b \sulut{K/\gamma(s)}^{1/\sulut{p(s)-1}} \der s = B- A.
\end{equation}
Writing $m = B-A$, we have implicitly defined a function $K_m= K  \colon \R_+ \to \R_+$ by writing as $K_m$ the constant $K$ which satisfies the above equation with $m = B-A$.

The next lemma justifies the previous heuristic discussion, and also implies that the constant~$K_m$ is unique, since the minimizer is unique.
\begin{lemma}
Suppose that for a given $B-A = m \geq 0$, there exists a function $v \in W^{1,p}(I)$ satisfying the boundary values $(A,B)$, and a constant~$K_m$ for which the following equality is true almost everywhere in~$x$:
\begin{equation}
    \gamma(x) \sulut{v'(x)}^{p(x)-1} = K_m.
\end{equation}
Then $v$ is the unique minimizer of energy~\eqref{eq:normal_energy} with boundary values $A$ and $B$, and thus solves the variable exponent conductivity equation.
\end{lemma}

We use the same proof as Diening, Harjulehto, Hästö and R\r{u}\v{z}i\v{c}ka~\cite[lemma 13.1.4]{Diening:Harjulehto:Hasto:Ruzicka:2011}, but they consider a slightly different equation.

\begin{proof}
By lemma~\ref{lemma:unique_existence} there exists a unique minimizer.
As such, we only need to establish that $v$ minimizes the energy.

Let $w$ be such that $w-v \in W_0^{1,p\sulut{\cdot}}\sulut{I}$.
Since $\gamma(x) \sulut{v'(x)}^{p(x)-1}$ is constant almost everywhere,
\begin{equation}
    \int_a^b \gamma(x) \sulut{v'(x)}^{p(x)-1} \sulut{w'-v'} \der x = 0.
\end{equation}

We use the inequality
\begin{equation}
\frac{1}{p}\abs{y}^p \ge \frac{1}{p}\abs{z}^p + \abs{z}^{p-1}\sulut{y-z},
\end{equation}
which follows from the convexity of the differentiable function $y \mapsto \frac{1}{p}\abs{y}^p$.
The inequality implies
\begin{equation}
    \int_a^b \frac{\gamma\sulut{x}}{p\sulut{x}} \abs{w'}^{p\sulut{x}} \der x
    \geq \int_a^b \frac{\gamma\sulut{x}}{p\sulut{x}} \abs{v'}^{p\sulut{x}} \der x
    + \int_a^b \gamma\sulut{x} \abs{v'}^{p\sulut{x}-1}\sulut{w'-v'} \der x,
\end{equation}
which implies, since the last integral is zero, 
\begin{equation}
    \int_a^b \frac{\gamma\sulut{x}}{p\sulut{x}} \abs{w'}^{p\sulut{x}} \der x
    \geq \int_a^b \frac{\gamma\sulut{x}}{p\sulut{x}} \abs{v'}^{p\sulut{x}} \der x.
\end{equation}
\end{proof}

\begin{lemma} \label{lemma:m2K_bijection}
The map $m \mapsto K_m$ is a strictly increasing, continuous bijection.
\end{lemma}

\begin{proof}
It suffices to prove that the map
\begin{equation}
K \mapsto \int_a^b \sulut{K/\gamma(x)}^{1/\sulut{p(x)-1}} \der x
\end{equation}
is a strictly increasing surjection.

Since $1 < p < \infty$, the map is strictly increasing.
We have both $0 \mapsto 0$ and $\int_a^b \sulut{K/\gamma(x)}^{1/\sulut{p(x)-1}} \der x \to \infty$ as $K \to \infty$, since $\gamma$ and $1/(p-1)$ are positive.
The integrand is continuous with respect to $K$ for almost every $x\in I$, which implies continuity via dominated convergence, given the bounded interval and boundedness of $\gamma$ and $p$.
Hence, we have surjectivity.
\end{proof}

The DN map is $\Lambda_\gamma \colon \R^2 \to \R$, $\Lambda_\gamma \sulut{A,B} = \int_a^b \gamma \abs{u'}^p(x) \der x$, as established in section~\ref{sec:forward}.
Assuming $A \le B$ and inserting $u'$ gives the following lemma:
\begin{lemma}
\label{lemma:DN-px}
Suppose $A\leq B$. Then
\begin{equation}
\Lambda_\gamma(A,B) = \Lambda_\gamma(m) = \int_a^b \gamma^{-1/(p(x)-1)} K_m^{p(x)/(p(x)-1)} \der x,
\end{equation}
where $K_m$ is also a function of the conductivity~$\gamma$.
\end{lemma}

First we observe that we can recover $\int_a^b \gamma^{-1/(p(x)-1)} \der x$ from the Dirichlet to Neumann map as its unique fixed point.
\begin{lemma}
\label{lemma:fixed}
Suppose $B > A$.
\begin{itemize}
\item If $B-A > \int_a^b \gamma^{-1/(p(x)-1)} \der x$, then $\Lambda_\gamma(B-A) > B-A$. Also, $K>1$.
\item If $B-A = \int_a^b \gamma^{-1/(p(x)-1)} \der x$, then $\Lambda_\gamma(B-A) = B-A$. Also, $K=1$.
\item If $ B-A < \int_a^b \gamma^{-1/(p(x)-1)} \der x$, then $\Lambda_\gamma(B-A) < B-A$. Also, $K<1$.
\end{itemize}
\end{lemma}
\begin{proof}
The positive number $k$ is a fixed point of the DN map if and only if
\begin{align}
k = \Lambda_\gamma (k) &= \int_a^b \gamma^{-1/(p(x)-1)} K_k^{p(x)/(p(x)-1)} \der x \text{ where}\\
k & = \int_a^b \sulut{K_k/\gamma}^{1/\sulut{p(x)-1}} \der x,
\end{align}
which implies
\begin{equation*}
\int_a^b \gamma^{-1/(p(x)-1)} K_k^{p(x)/(p(x)-1)} \der x =
\int_a^b \gamma^{-1/\sulut{p(x)-1}}  K_k^{1/\sulut{p(x)-1}} \der x.
\end{equation*}
This is true if and only if $K_k = 1$, since $\gamma > 0$ and $p > 0$.

If $K = 1$, then by equation~\eqref{eq:K_dirichlet} we have $B-A = \int_I \gamma^{-1/(p(x)-1)} \der x$.
By bijectivity of $m \mapsto K_m$, there only exists one $m$ with $K_m=1$.
This proves the middle point of the claim.

If $m > k$, then $K_m > 1$, and hence
\begin{equation}
K_m^{p(x)/(p(x)-1)} >  K_m^{1/\sulut{p(x)-1}},
\end{equation}
which implies
\begin{equation}
\begin{split}
& \Lambda_\gamma(B-A) = \int_a^b \gamma^{-1/(p(x)-1)} K_m^{p(x)/(p(x)-1)} \der x \\
> & \int_a^b \gamma^{-1/\sulut{p(x)-1}}  K_m^{1/\sulut{p(x)-1}} \der x = B-A.
\end{split}
\end{equation}
The same argument with reversed inequalities holds when $m < k$.
\end{proof}
By using, for example, the half-interval search we get:
\begin{corollary}
\label{lemma:av_sigma}
The quantity
\begin{equation}
\int_I \gamma^{-1/(p(x)-1)} \der x
\end{equation}
can be recovered from the Dirichlet to Neumann map.
\end{corollary}


The next remark concerns the inverse problem with additional interior data of the type that can, under some idealizations, be recovered using hybrid imaging methods such as ultrasound mediated electrical impedance tomography (UMEIT, also called AET for acousto-electric tomography), conductivity density impedance imaging (CDII) and magnetic resonance electrical impedance tomography (MREIT).~\cite{Bal:2013,Kuchment:Steinhauer:2012,Kwon:Woo:Yoon:Seo:2002}
\begin{remark}[Interior data and variable exponent]
\label{remark:interior_1d}
If we have knowledge of interior power data of type $\gamma \abs{u'}^{r(x)}$, where $0 \le r(x) < \infty$, then the conductivity can be recovered at all points where $p(x)-r(x)\neq 1$.
Indeed, a simple calculation gives
\begin{equation}
\gamma \abs{u'}^{r(x)} = \gamma \sulut{K/\gamma(x)}^{r(x)/(p(x)-1)} = \sulut{\gamma(x)}^{(p(x)-r(x)-1)/(p(x)-1)} K^{r(x)/(p(x)-1)}.
\end{equation}
We can choose the Dirichlet data $B-A$ so that $K=1$ by lemma~\ref{lemma:m2K_bijection}.
Hence, $\gamma$ can be recovered whenever it has a nonzero exponent, or, equivalently, whenever $p(x)-r(x) \neq 1$.
\end{remark}
This generalizes a result of Straub~\cite[chapter 3]{Straub:2016}, which was for $p \equiv 2$.


\subsection{Identification at extremes}
\label{sec:limits}

Next we recover the value of $\gamma$ at the points where $p(x)$ takes its maximum or minimum value.
First write
\begin{equation}
q(x) = \frac{p(x)}{p(x)-1} \text{ and } f(x) = \sulut{\gamma(x)}^{-1/(p(x)-1)}.
\end{equation}
These are both injective mappings of $p(x)$ and $\gamma(x)$, respectively, and $q$ is the Hölder dual exponent of $p$.
The maxima of $p$ correspond to the minima of $q$ and vice versa.

\begin{lemma}
\label{lemma:K_DN_lim}
Suppose the exponent~$q$ reaches its essential supremum (respectively infimum) value~$q^+$ ($q^-$) on a set $Q_+$ ($Q_-$) of positive measure.
Then
\begin{align}
\lim_{K \to \infty} K^{-q^+}\int_a^b f(x) K^{q(x)} \der x &= \int_{Q_+} f(x) \der x \text{ and} \\
\lim_{K \to 0} K^{-q^-}\int_a^b f(x) K^{q(x)} \der x &= \int_{Q_-} f(x) \der x.
\end{align}
\end{lemma}

\begin{proof}
For the maximum, by monotone (or dominated) convergence
\begin{equation}
\int_{I \setminus Q_+} K^{q(x)-q^+} f(x) \der x \to 0,
\end{equation}
since $q(x)-q^+ < 0$ on the set.
The integral over $Q_+$ gives what we claim in the lemma.
The argument for the minimum has precisely the same idea.
\end{proof}

Unfortunately, $K_m$ is not something we can recover from the measurements.
We define an auxiliary variable~$\ol K_m$, which corresponds to conductivity one and thus is characterized by the equation
\begin{equation}
m = \int_a^b \ol K^{1/(p(x)-1)}_m \der x
\end{equation}
and can be calculated without knowing the true conductivity.
We write as $c_m$ constants that satisfy the inequality
%
\begin{equation}
1/c_m \leq K_m/\ol K_m \leq c_m.
\end{equation}
%
Let $K_m^l$ correspond to the constant weight~$\max\sulut{1,\esssup_{x \in I}f(x)}$ and $K_m^u$ to $\min\sulut{1,\essinf_{x \in I}f(x)}$.
That is, we have
\begin{align}
m &= \int_a^b \max\sulut{1,\esssup_{x \in I}f(x)} \sulut{K_m^l}^{1/(p(x)-1)} \der x \text{ and}\\
m &= \int_a^b \min\sulut{1,\essinf_{x \in I}f(x)} \sulut{K_m^u}^{1/(p(x)-1)} \der x.
\end{align}
This implies $K_m^l \leq K_m \leq K_m^u$ and $K_m^l \leq \ol K_m \leq K_m^u$,
and thereby $K_m/ \ol{K_m}$ and its inverse are bounded by $K_m^u/K_m^l$.
To bound this we use a mean value theorem:
\begin{lemma}
\label{lemma:mvt}
Suppose $p \in L^{\infty}([a,b])$  and suppose $h \colon A \to \R$, with the range $[\essinf p, \esssup p] \subseteq A$, is a continuous function with
\begin{equation}
\int_a^b h \sulut{p(x)} \der x = 0.
\end{equation}
Then there exists a real number $p^* \in [\essinf p, \esssup p]$ such that
\begin{equation}
\int_a^b h(p^*) \der x = 0.
\end{equation}
\end{lemma}

\begin{proof}
If no $p^*$ with $h(p^*) = 0$ existed, then by continuity of $h$, $h(p)$ would be either positive for all $p \in [\essinf p, \esssup p]$, or negative for all of them.
This contradicts the assumption that the integral of $h(p(x))$ is zero.
\end{proof}

\begin{lemma}
\label{lemma:olK_K_bounded}
For all $m \in \R_+$ there exists $p^* \in [p^- , p^+]$ such that we can choose
\begin{equation}
c_m \leq \sulut{\frac{\max\sulut{1,\esssup_{x \in I}f(x)}}{\min\sulut{1,\essinf_{x \in I}f(x)}}}^{p^* -1}.
\end{equation}
\end{lemma}

\begin{proof}
We define $h$ by
\begin{equation}
h(p) = \sulut{K_m^u}^{1/(p-1)}\min\sulut{1,\essinf_{x \in I}f(x)} - \sulut{K_m^l}^{1/(p-1)}\max\sulut{1,\esssup_{x \in I}f(x)}
\end{equation}
and use the mean value lemma (lemma~\ref{lemma:mvt}).
The claim follows from
\begin{equation}
\begin{split}
&\int_a^b \max\sulut{1,\esssup_{x \in I}f(x)} \sulut{K_m^l}^{1/(p^*-1)} \der x  \\
= &\int_a^b \min\sulut{1,\essinf_{x \in I}f(x)} \sulut{K_m^u}^{1/(p^*-1)} \der x.
\end{split}
\end{equation}
\end{proof}

We have thus established that the ratio~$K_m/\ol K_m$ is bounded uniformly in $m$, since $p^*$ is bounded.
We use this information to determine the limit of the ratio as $m^{\pm1}\to \infty$.

\begin{lemma}
\label{lemma:lim_olK_K}
Suppose $Q_+$ (respectively $Q_-$) has positive measure.
Then
\begin{align}
\lim_{m \to \infty} \sulut{\frac{\ol K_m}{K_m}}^{1/(p^- -1)} & = \mint{-}_{Q^+} f(x) \der x \\
\lim_{m \to 0} \sulut{\frac{\ol K_m}{K_m}}^{1/(p^+ -1)} & = \mint{-}_{Q^-} f(x) \der x.
\end{align}
\end{lemma}

\begin{proof}
First we prove
\begin{align}
\int_{Q_+} f(x) \der x K_m^{1/(p^- -1)} &= \abs{Q_+} \ol K_m^{1/(p^- -1)} + o_{m \to \infty}\sulut{K^{1/(p^- -1)}_m + \ol K^{1/(p^- -1)}_m}
\end{align}
and
\begin{align}
\int_{Q_-} f(x) \der x K_m^{1/(p^+ -1)} &= \abs{Q_-} \ol K_m^{1/(p^+ -1)} + o_{m \to 0}\sulut{K^{1/(p^+ -1)}_m + \ol K^{1/(p^+ -1)}_m}.
\end{align}
For $m \to \infty$ we calculate
\begin{equation}
\begin{split}
\int_{Q_+} f(x)  \der x & K_m^{1/(p^- -1)} = \int_{Q_+} f(x)  K_m^{1/(p(x) -1)} \der x \\
&= \int_a^b f(x)  K_m^{1/(p(x) -1)} \der x + o_{m \to \infty}\sulut{K^{1/(p^- -1)}_m} \\
&= \int_a^b \ol K_m^{1/(p(x) -1)} \der x + o_{m \to \infty}\sulut{K^{1/(p^- -1)}_m} \\
&= \abs{Q_+} \ol K_m^{1/(p^- -1)} + o_{m \to \infty}\sulut{K^{1/(p^- -1)}_m + \ol K^{1/(p^- -1)}_m}.
\end{split}
\end{equation}

Next we divide by $\abs{Q_+}K_m^{1/(p^- -1)}$ and use 
lemma~\ref{lemma:olK_K_bounded} to get the claim.
The argument when $m \to 0$ is essentially the same.
\end{proof}

Theorem~\ref{thm:pos_meas} directly follows from the following proposition.
\begin{proposition}
Suppose $Q_+$ or $Q_-$ is of positive measure.
Then
\begin{align}
\lim_{m \to \infty} \ol K_m^{-q^+} \frac{1}{\abs{Q_+}} \Lambda_\gamma(m) & = \sulut{\mint{-}_{Q_+}f(x) \der x}^{-\sulut{p^- -1}} \text{ or} \\
\lim_{m \to 0} \ol K_m^{-q^-} \frac{1}{\abs{Q_-}} \Lambda_\gamma(m) & = \sulut{\mint{-}_{Q_-}f(x) \der x}^{-\sulut{p^+ -1}}.
\end{align}
\end{proposition}

\begin{proof}
\begin{equation}
\begin{split}
& \ol K_m^{-p^-/(p^- - 1)} \frac{1}{\abs{Q_+}} \Lambda_\gamma(m) \\
&= \sulut{\frac{\ol K_m}{K_m}}^{-q^+}K^{-q^+}_m \frac{1}{\abs{Q_+}} \int_a^b f(x)K^{q(x)}_m \der x.
\end{split}
\end{equation}
By lemma~\ref{lemma:K_DN_lim} we have, as $m \to \infty$,
\begin{equation}
K^{-q^+}_m \frac{1}{\abs{Q_+}} \int_a^b f(x)K^{q(x)}_m \der x \to \mint{-}_{Q^+} f(x) \der x.
\end{equation}
The proof then follows from lemma~\ref{lemma:lim_olK_K}.
\end{proof}

\subsection{Characterization of recognizable functions}
\label{sec:characterization}

In this section we first characterize the space of functions with which we pair the unknown function $f$ in the DN~map
\begin{equation}
    m \mapsto \int_I f(x) K_m^{q(x)} \der x.
\end{equation}
We write $k_m = \log \sulut{K_m}$ and consider the set of functions
\begin{equation}
    S = \joukko{ \exp{\sulut{k_m q(x)}} ; m\in \R }. 
\end{equation}

Since $S$ is closed under pointwise multiplication, any product of linear combinations of elements from $S$ remains a linear combination of elements from $S$.
In particular, the space~$\Span\sulut{S}$ is an algebra.

In the special case that $q$ is continuous and injective, the Stone-Weierstrass theorem implies that $\Span\sulut{S}$ is dense in $C(I)$ with the usual topology of uniform convergence. Since the continuous functions are dense and continuously embedded in $L^2(I)$ (the interval~$I$ is bounded), $\Span\sulut{S}$ must also be dense in $L^2(I)$ by approximating in $L^2$-norm and choosing a diagonal sequence from $\Span \sulut{S}$. This result we shall now generalize to the setting where $q$ is merely measurable by using the multiplicative system theorem~\cite[theorem~21]{Dellacherie:Meyer:1978}.
\begin{theorem}[Multiplicative system theorem]
\label{thm:multisystem}
Suppose $H$ is a vector space of real-valued bounded measurable functions on a measurable space $X$.
Suppose $H$ contains constant functions and is closed under the pointwise convergence of uniformly bounded increasing sequences of functions.
Let $M \subseteq H$ be closed under pointwise multiplication, and let $\mathcal{G}$ be the $\sigma$-algebra generated by $M$.

Then $H$ contains all bounded $\mathcal{G}$-measurable functions.
\end{theorem}

Suppose $r > 1$ is a positive real number and $s > 1$ is its Hölder conjugate.
Note that, when $p$ and $q$ are Hölder conjugates,
$\sigma(p) = \sigma(q)$,
since
the map taking $p$ to $q$ is a homeomorphism from $]1,\infty[$ to itself. Hence $p$ and $q$ generate the same $\sigma$-algebra.

\begin{lemma}
\label{lemma:span}
\begin{equation}
\overline{\Span(S)}^{L^r(I)} = L^r\sulut{I, \sigma(p)}.
\end{equation}
\end{lemma}

The proof follows a proof of Nathaniel Eldredge~\cite{Eldredge:2018} for a similar lemma.
We omit the space $L^r$ from the notation of the closure.
\begin{proof}
The $\sigma$-algebra generated by $\Span\sulut{S}$ is exactly the sigma-algebra $\sigma(p)$.
One consequence is that $\Span\sulut{S}$ is a subspace of $L^r\sulut{I, \sigma(p)}$.

First, we show $\ol{\Span\sulut{S}} \subseteq L^r\sulut{I, \sigma(p)}$.
Pick an element in the $L^r$-closure. Then there is a Cauchy sequence in $\Span(S)$ with this point as its limit. Any Cauchy sequence in $L^r(I)$ has an almost everywhere convergent subsequence~\cite[theorems 3.11 and 3.12]{Rudin:2006}, and since these functions are measurable with respect to $\sigma(p)$, so is the limit. This gives the first inclusion.

The reverse inclusion will follow from the multiplicative systems theorem.
Define $M=\Span\sulut{S}$ and let $H$ consist of all bounded measurable functions belonging to the equivalence classes of functions inside $\ol{\Span\sulut{S}} \cap L^\infty(I)$. We note that $M\subset H$ is closed under pointwise multiplication. 

\begin{enumerate}

    \item $H$ contains constant functions, since $1\in S \subset H$ by taking $k_m = 0$.
    
    \item $H$ is closed under pointwise convergence of uniformly bounded increasing sequences. To see this, take such a sequence $(f_j)_{j=1}^\infty$ in $H$ converging pointwise to a measurable function $f$. Then $f$ is bounded and $|f-f_j|^r$ converges pointwise to zero. Uniform boundedness and the dominated convergence theorem together imply that $f_j \to f$ in the $L^r$-norm as $j \to \infty$. Thus the equivalence class of $f$ is in $\ol{\Span\sulut{S}}$. Hence $f\in H$.
    
\end{enumerate}

Theorem~\ref{thm:multisystem} ensures $H$ contains all bounded $\sigma(p)$-measurable functions. By this construction, $L^r\sulut{I, \sigma(p)} \cap L^\infty\sulut{I, \sigma(p)} \subset \ol{\Span\sulut{S}}\cap L^\infty\sulut{I}$ holds.
Note that we here use that the sigma-algebra generated by $M$ is actually $\sigma(p)$.
Take $h \in L^r\sulut{I,\sigma(p)}$ and let $h$ also signify a $\sigma(p)$-measurable representative. 
Construct a sequence $(h_j)_{j=1}^\infty$ by setting
\begin{equation}
h_j = \max\joukko{-j,\min\joukko{h,j}}.
\end{equation}
These are bounded and $\sigma(p)$-measurable. By the above, they belong to $H$.
The dominated convergence theorem implies $h_j \to h$ as $j \to \infty$ in the $L^r$-norm. Since each $h_j$ is in some equivalence class of $\ol{\Span\sulut{S}}$, then $h \in \ol{\Span\sulut{S}}$ also. This gives the reverse inclusion.
\end{proof}

Above, we demonstrated that the functions $\exp(k_m q)$, spanning a dense subspace of $L^r(I,\sigma(p))$, suffice to determine $f\in L^s(I,\sigma(p))$ uniquely. Hence $f$ can, in principle, be recovered from measurements of the DN~map across all $m\in \mathbb{R}$, provided that it belongs to this space. In general it belongs to the space $L^\infty(I)$, as $\gamma$ and $p$ are bounded, and thereby to all the $L^r$ spaces.

When $q$ is continuous and increasing, we are unable to recover $f$ in sets where $q$ is constant. This is because $K_m^{q\sulut{\cdot}}$ restricted to any such set remains constant upon varying $m\in \mathbb{R}$, so testing against these yields no insight beyond the average of $f$ inside such a set. On the other hand, if $q$ is even on an interval symmetric about the origin, then we can only hope to determine $f$ up to its even part, because all $\exp(k_m q)$ are even in this case.

In abstract language, what we have determined is the projection of $f \in L^\infty(I) \subset L^2(I)$ onto the (complete and therefore closed) subspace $L^2(I, \sigma(p))$. This projection can also be understood in terms of the conditional expectation $\E(f \, | \, \sigma(p))$ given $\sigma(p)$ as the algebra of observable events~\cite[theorem 3.2.6]{Bobrowski:2005}. 

Let $P \colon L^2(I) \to L^2(I, \sigma(p))$ be the orthogonal projection onto $L^2(I, \sigma(p))$. We formulate the above as a statement about the reconstructibility of $Pf$ from a collection of measurements of the DN~map:

\begin{proposition} \label{prop:characterization}
There exists an orthonormal sequence $\{ s_n \}_{n=1}^\infty \subset \Span\sulut{S}$ determining the projection $Pf = \sum_{n=1}^\infty \langle f, s_n \rangle s_n$ for any $f\in L^2(I)$.
\end{proposition}
\begin{proof}
Since $L^2(I,\sigma(p))$ identifies with a subspace of $L^2(I)$, it is also separable.
By lemma~\ref{lemma:span} there is a linearly independent countable dense subset of $\Span\sulut{S}$. Orthonormalization by the Gram-Schmidt process gives the vectors.
\end{proof}

Every $s_n$ above is some finite linear combination of the functions $\exp(k_m q)$, and the coefficients are determined from finite combinations of $\langle f, \exp(k_m q) \rangle$, the measurements of the DN map. 

By the above, we only need countably many measurements of the DN map, but since the functions $s_n$ depend implicitly on the unknown conductivity~$\gamma$, we have no way to determine which Dirichlet data~$m$ to use beforehand. It is not possible to explicitly reconstruct $Pf$ using proposition~\ref{prop:characterization}.

Functions that can not be detected are simply those belonging to $\ker(P)$, and we have the following simple characterization: 

\begin{proposition} \label{prop:kernel}
The kernel $\ker(P)$ consists of all those functions $f\in L^2(I)$ that integrate to zero on every $\sigma(p)$-set.
\end{proposition}
\begin{proof}
A function with the stated properties is in $\ker(P)$ by proposition~\ref{prop:characterization}, because any $s_n$ can be approximated in $L^2(I,\sigma(p))$ by using simple functions. Conversely, if $Pf=0$, then the integral of $f$ taken over any $\sigma(p)$-set is zero, because $\Span\{s_n\}_{n=1}^\infty$ can approximate any characteristic of a $\sigma(p)$-set. 
\end{proof}

\begin{example}
If $I$ is symmetric around the origin and $p$ is an even function, then all the $\sigma(p)$-sets are also symmetric about the origin, and by the above, all odd functions on $I$ must be in $\ker(P)$.
\end{example}

Finally, in terms of the conductivity $\gamma$, which is the function of interest, it is natural to use $P$ to define a nonlinear mapping
\begin{align*}
\widetilde{P} : L^2(I)\cap L^\infty_+ (I) \to L^2(I,\sigma(p))\cap L^\infty_+ (I,\sigma(p))
:
\gamma \mapsto \left(P(\gamma^{-\frac{1}{p-1}})\right)^{-(p-1)},
\end{align*}
To see that this is well-defined, we define, via representatives $g$, the mapping
\begin{align*}
\Phi :  L^2(I)\cap L^\infty_+ (I) \to  L^2(I)\cap L^\infty_+ (I) :  g \mapsto g^{-\frac{1}{p-1}},
\end{align*}
and put $\widetilde{P}= \Phi^{-1}P\Phi$.
Since we have the bounds $1 < p^- < p(x) < p^+ < \infty$, then $\Phi$ maps $L^\infty_+(I)$ to itself, and hence to $L^2(I)$.
We note that $\Phi$ is invertible.
As an aside, we mention that $P$ and $\widetilde{P}$ are topologically conjugate~\cite[section 4.7]{Meiss:2017} as continuous maps on $L^\infty_+\cap L^2$ in the topology of $L^\infty_+$, since $\Phi$ is a homeomorphism from $L^\infty_+$ to itself.

\begin{lemma}
Provided that $g$ is $\sigma(p)$-measurable, then $g^{-\frac{1}{p-1}}$ is also $\sigma(p)$-measurable.
The inverse $g^{-(p-1)}$ has the same property.
\end{lemma}
\begin{proof}
We consider the function
\begin{equation}
 x \mapsto  (g(x))^{-\frac{1}{p(x)-1}}.
\end{equation}
There exist uniformly bounded sequences of $\sigma(p)$-measurable simple functions, denoted $(g_j)$ and $(h_j)$, converging pointwise to $g$ and the exponent, respectively.
Then $(g_j^{h_j})$ converges pointwise to $g^{-\frac{1}{p-1}}$.
The proof is similar for $g^{-(p-1)}$.
\end{proof}

The projection $P$ also maps $L^\infty_+(I)$ into $L^\infty_+(I,\sigma(p))$. Indeed,
 observe that $\esssup (Pg) \leq \esssup (g)$ for any $g \in L^\infty(I)$, since the averages of $Pg$ and $g$ over any $\sigma(p)$ set must be equal, and the set $\{x\in I \,;\, Pg(x) > \esssup (g) + \eps \}$ has measure zero for every $\eps>0$. Similarly $\essinf (g) \leq \essinf (Pg)$.

Therefore $\widetilde{P}= \Phi^{-1}P\Phi$ is well-defined with the desired mapping properties.
It inherits the projection property $\widetilde{P}\circ \widetilde P = \widetilde{P}$ from $P$.
This proves theorem~\ref{thm:characterization}.

\subsection{Derivatives of the Dirichlet to Neumann map}
\label{sec:derivatives}

In this section our goal is to give a constructive, if ill-posed and inconvenient, alternative to the nonconstructive result in previous section~\ref{sec:characterization}.
Furthermore, the result here is local in the sense that we only need to know the Dirichlet-to-Neumann map in a neighbourhood of the value $m$ for which $K(m)=1$, which we will write as $k$.
That is, we have $K(k) = 1$.
We also write $K_m $ as $ K(m)$ to emphasize the dependence on $m$.

We will state explicit formulae for high order derivatives of inverse and compound functions.
The formulae are used to calculate $\frac{\der^n m}{\der K^n}|_{m=k}$ in terms of derivatives of the DN map, which then allows calculating $\int_I f(x) \sulut{\frac{1}{p(x)-1}}^n \der x$ and proceeding as in the previous section.

Recall that
\begin{equation}
\Lambda_\gamma(m) = \int_I f(x) \sulut{K(m)}^{q(x)} \der x,
\end{equation}
which we now consider as a function of $K$:
\begin{equation}
\Lambda_\gamma(K) = \int_I f(x) K^{q(x)} \der x.
\end{equation}
Also recall
\begin{equation}
m = \int_I f(x) K^{1/\sulut{p(x)-1}} \der x ,
\end{equation}
whence
\begin{equation}
\frac{\der^j m}{\der K^j} = \int_I f(x) K^{1/\sulut{p(x)-1}-j} \prod_{l=0}^{j-1} \sulut{1/\sulut{p(x)-1}-l} \der x.
\end{equation}

We record the following formula for higher order derivatives of inverse functions~\cite{Kaneiwa:2016}.
The facts about the indices follow by elementary manipulation.
\begin{lemma}
\label{lemma:higher_inverse}
Suppose $n \in \Z_+$, for all $j$ we have $s_j \in \N$, $m$ is a $C^n$-function of $K$, and $\frac{\der m}{\der K} \neq 0$ on an open interval.
Then, on that interval,
\begin{equation}
\frac{\der^n K}{\der m^n} = \frac{\sulut{-1}^{n-1}}{\sulut{\frac{\der m}{\der K}}^{2n-1}}
\sum_{
\substack{s_1+s_2+ \cdots = n-1 \\
                               1\cdot s_1 + 2\cdot s_2 + \cdots = 2n-2}}
\frac{\sulut{-1}^{s_1}\sulut{2n-s_1-2}!\prod_{j \in \Z_+}\sulut{\frac{\der^j m}{\der K^j}}^{s_j}}{\prod_{j=2}^{\infty}\sulut{j!}^{s_j}s_j!}.
\end{equation}
Furthermore, we have:
\begin{itemize}
    \item $s_j \neq 0$ implies $j \leq n$.
    \item $s_n \neq 0$ implies the index tuple $(s_1,s_2,\ldots,s_{n-1},s_{n}) = \sulut{n-2,0\ldots,0,1}$ (and $s_1=1$ when $n=1$).
\end{itemize}
\end{lemma}
We note that the derivative $ \frac{\der m}{\der K}$ is indeed nonzero, though higher order derivatives might not be, if $p$ is constant.

Next we state the well known Faà di Bruno's formula~\cite{Johnson:2002}.
\begin{lemma}[Faà di Bruno's formula]
For $n$~times continuously differentiable $K \mapsto \Lambda_\gamma$ and $m \mapsto K$
\begin{equation}
\frac{\der^n}{\der m^n} \Lambda_\gamma \sulut{K(m)}
=\sum \frac{n!}{k_1!\,k_2!\,\cdots\,k_n!}\cdot
 \Lambda_\gamma^{(k_1+\cdots+k_n)}(K(m))\cdot
\prod_{j=1}^n\left(\frac{K^{(j)}(m)}{j!}\right)^{k_j},
\end{equation}
where the sum is over $n$-tuples of numbers $k_j$ with
\begin{equation}
\sum_{j=1}^n jk_j = n.
\end{equation}
\end{lemma}

We have
\begin{equation}
\frac{\der^j \Lambda_\gamma}{\der K^j}(K) = \int_I f(x) K^{q(x)-j} \prod_{l=0}^{j-1} \sulut{q(x)-l} \der x.
\end{equation}

Hence, the derivative of the DN map with respect to $m$ at $m=k$ is
\begin{equation}
\begin{split}
\frac{\der \Lambda_\gamma \sulut{K(m)}}{\der m} |_{m=k}
&=
 \frac{\int_I f(x) \frac{p(x)}{p(x)-1} \der x}{\int_I f(x) \frac{1}{p(x)-1} \der x}.
 \end{split}
\end{equation}
We observe that since $p(x) >p^- > 1$, the derivative is always greater than 1.
We already know $\int_I f(x) \der x$ due to corollary~\ref{lemma:av_sigma}.
Using $\frac{p}{p-1} = 1 + \frac{1}{p-1}$ we get
\begin{equation}
\int_I\frac{1}{p(x)-1}f(x) \der x = \int_I f(x) \der x \sulut{\frac{\der}{\der m} \Lambda_\gamma \sulut{1} - 1}^{-1},
\end{equation}
which is well-defined, since the DN map is strictly greater than one.
Hence, we now know the dual pairing of $f$ against $1$ and $1/ (p(x)-1)$.
This proves the following lemma:
\begin{lemma}
\label{lemma:pair_one}
\begin{equation}
\int_I f(x) \frac{1}{p(x)-1} \der x
\end{equation}
can be recovered from
\begin{equation}
\Lambda_\gamma (K(k)) \text{ and } \frac{\der \Lambda_\gamma (K(m))}{\der m}|_{m=k}.
\end{equation}
\end{lemma}
We prove the analogous claim for higher powers of $\frac{1}{p(x)-1}$ by induction.

\begin{proposition}
Suppose $n \in \N$.
The quantity
\begin{equation}
\int_I f(x) \sulut{\frac{1}{p(x)-1}}^n \der x
\end{equation}
can be recovered from
\begin{equation}
\frac{\der^l \Lambda_\gamma \sulut{K(m)}}{\der m^l}|_{m=k} \text{ for } l \in \joukko{0,\ldots,n}.
\end{equation}
\end{proposition}

\begin{proof}
Corollary~\ref{lemma:av_sigma} establishes the case~$n=0$ and lemma~\ref{lemma:pair_one} the case~$n=1$.
We proceed by strong induction and suppose $n > 1$.
That is, we assume that the following integrals are known for all $j < n$:
\begin{equation}
\int_I f(x) \sulut{\frac{1}{p(x)-1}}^{j} \der x.
\end{equation}

In Faà di Bruno's formula for the derivative of order~$n$ there is only one $n$-tuple where
$\frac{\der^n \Lambda_\gamma}{\der K^n}$ appears,
namely $\sulut{n,0,\ldots,0}$.
Likewise, in the formula for the higher order inverse function, there is only one term with derivative of order~$n$, $\sulut{n-2,0,\ldots,0,1}$.
This latter term appears in Faà di Bruno's formula precisely when the $n$-tuple is $\sulut{0,\ldots,1}$.
The other summands in the formulae only contain known terms, not
\begin{equation}
\int_I f(x) \sulut{\frac{1}{p(x)-1}}^n \der x.
\end{equation}

By Faà di Bruno's formula we write the derivative as
\begin{equation}
\begin{split}
\frac{\der^n \Lambda_\gamma}{\der m^n}|_{m=k} &  = 
\frac{\der^n \Lambda_\gamma}{\der K^n}|_{K=1} \cdot
\left( \frac{\der m}{\der K}|_{K=1} \right)^{-n} 
+ \frac{\der \Lambda_\gamma}{\der K}|_{K=1} \frac{\der^n K}{\der m^n}|_{m=k}
+ S_1 ,
\end{split}
\end{equation}
where $S_1$ is known by the induction hypothesis.
Rewriting, we have
\begin{equation}
\begin{split}
\frac{\der^n \Lambda_\gamma}{\der m^n}|_{m=k} - S_1 
&=
\frac{\der^n m}{\der K^n}|_{K=1}
\sulut{ \int_I f(x) \frac{1}{p(x)-1} \der x}^{-n} \\
&+ \int_I f(x) \frac{p(x)}{p(x)-1} \der x \frac{\der^n K}{\der m^n}|_{m=k},
\end{split}
\end{equation}
where, by lemma~\ref{lemma:higher_inverse},
\begin{equation}
\begin{split}
\frac{\der^n K}{\der m^n}&|_{m=k} = -\sulut{\frac{\der m}{\der K}|_{K=1}}^{-(n+1)}\frac{\der^n m}{\der K^n}|_{K=1} 
+ S_2,
\end{split}
\end{equation}
where $S_2$ is known by the induction hypothesis.
This gives
\begin{equation}
\begin{split}
\frac{\der^n \Lambda_\gamma}{\der m^n}|_{m=k} - S_1 &- S_2\int_I f(x) \frac{p(x)}{p(x)-1} \der x \\
&=
\frac{\der^n m}{\der K^n}|_{K=1}
\sulut{ \int_I f(x) \frac{1}{p(x)-1} \der x}^{-n} \\
& -\sulut{\frac{\der m}{\der K}|_{K=1}}^{-(n+1)}\frac{\der^n m}{\der K^n}|_{K=1} \int_I f(x)  \frac{p(x)}{p(x)-1} \der x.
\end{split}
\end{equation}
Using $p/(p-1) = 1 + 1/(p-1)$ we get
\begin{equation}
\begin{split}
\frac{\der^n \Lambda_\gamma}{\der m^n}|_{m=k} - S_1 &- S_2\int_I f(x) \frac{p(x)}{p(x)-1} \der x \\
&= -\sulut{\frac{\der m}{\der K}|_{K=1}}^{-(n+1)}\frac{\der^n m}{\der K^n}|_{K=1} \int_I f(x) \der x.
\end{split}
\end{equation}
Because everything else on the right hand side is positive, this gives an explicit formula for
\begin{equation}
\frac{\der^n m}{\der K^n}|_{K=1} = \int_I f(x) \prod_{l=0}^{n-1} \sulut{1/\sulut{p(x)-1}-l} \der x,
\end{equation}
where $\prod_{l=0}^{n-1} \sulut{1/\sulut{p(x)-1}-l}$ is a polynomial in $1/(p(x)-1)$ of order $n$ and with leading coefficient 1 for its highest order term.
By the induction hypothesis, the dual pairings of all the terms but the highest order one with $f$ are known.
This proves the claim.
\end{proof}

The following lemma is very similar to lemma~\ref{lemma:span} and \cite{Brander:Ilmavirta:Tyni}.
The proof also follows the proof of Nathaniel Eldredge~\cite{Eldredge:2018}.
\begin{lemma}
\label{lemma:polynomials}
Let $1 < r < \infty$. Then
\begin{equation}
\overline{\Span\sulut{\joukko{\sulut{\frac{1}{p(x)-1}}^n; n \in \N}}}^{L^r(I)} = L^r\sulut{I, \sigma(p)}.
\end{equation}
\end{lemma}

\begin{proof}
We first observe that $p \mapsto 1/(p-1)$ is a homeomorphism from $]1,\infty[$ to $]0,\infty[$.
This implies that $\sigma(p) = \sigma (1/(p-1))$.
Since $p$ is bounded away from one and infinity, $1/(p-1)$ is likewise bounded and hence
\begin{equation}
\Span\sulut{\joukko{\sulut{\frac{1}{p(x)-1}}^n; n \in \N}} \subseteq L^r\sulut{I}.
\end{equation}
The same argument as in the proof of lemma~\ref{lemma:span} establishes
\begin{equation}
\Span\sulut{\joukko{\sulut{\frac{1}{p(x)-1}}^n; n \in \N}} \subseteq L^r\sulut{I,\sigma(p)}.
\end{equation}

The proof for the inclusion
\begin{equation}
 L^r\sulut{I,\sigma(p)} \subseteq \Span\sulut{\joukko{\sulut{\frac{1}{p(x)-1}}^n; n \in \N}}
\end{equation}
is essentially the same as in lemma~\ref{lemma:span}, though in this case $n=0$ gives that the function~$1 \in H$.
In particular, multiplying two linear combinations of polynomials (of $1/(p-1)$) still gives a polynomial.
\end{proof}




Proving theorem~\ref{thm:characterization} proceeds as in the non-constructive case.

\bibliographystyle{plain}
\bibliography{math}

\end{document}